\documentclass{amsart}
\usepackage{amsmath}
\usepackage{amssymb}
\usepackage{enumerate}
\usepackage{mathrsfs}
\usepackage{color}
\usepackage{tikz}

\newtheorem{theorem}{Theorem}[section]
\newtheorem{lemma}[theorem]{Lemma}
\newtheorem{definition}[theorem]{Definition}
\newtheorem{example}[theorem]{Example}

\newtheorem{proposition}[theorem]{Proposition}

\newtheorem{remark}[theorem]{Remark}

\numberwithin{equation}{section}

\def\R{{\mathbb{R}}}

\def\N{{\mathbb{N}}}

\def\x{{\mathbf{x}}}

\def\a{{\boldsymbol{\alpha}}}

\def\b{{\boldsymbol{\beta}}}
\def\bv{{\boldsymbol{v}}}
\def\bu{{\boldsymbol{u}}}
\def\A{{\mathscr{A}}}

\def\supp{\hbox{\rm{supp}}}

\def\SONC{\hbox{\rm{SONC}}}

\def\int{\hbox{\rm{int}}}
\def\New{\hbox{\rm{New}}}
\def\Conv{\hbox{\rm{conv}}}

\begin{document}

\title{On Supports of Sums of Nonnegative Circuit Polynomials}
\thanks{This work was supported partly by NSFC under grants 61732001 and 61532019.}
\author{Jie Wang}
\address{Jie Wang\\ School of Mathematical Sciences, Peking University}
\email{wangjie212@pku.edu.cn}
\date{\today}

\begin{abstract}
In this paper, we prove that every SONC polynomial decomposes into a sum of nonnegative circuit polynomials with the same support, which reveals the advantage of SONC decompositions for certifying nonnegativity of sparse polynomials compared with the classical SOS decompositions. By virtue of this fact, we can decide $f\in\SONC$ through relative entropy programming more efficiently.
\end{abstract}

\maketitle
\bibliographystyle{amsplain}

\section{Introduction}
A real polynomial $f\in\R[\x]=\R[x_1,\ldots,x_n]$ is called a {\em nonnegative polynomial} if its evaluation on every real point is nonnegative. Certifying the nonnegativity of a polynomial $f$ is a central problem of real algebraic geometry and also has important applications to polynomial optimization problems. The classical method for this problem is writing $f$ as a sum of squares of polynomials (SOS), and then $f$ is obviously nonnegative. The key idea of this method is representing $f$ as a sum of a certain class of nonnegative polynomials whose nonnegativity is easy to check.

Recently in \cite{iw}, Iliman and Wolff introduced the concept of sums of nonnegative circuit polynomials as a substitute of sums of squares of polynomials to represent nonnegative polynomials. A polynomial $f$ is called a {\em circuit polynomial} if it is of the form
\begin{equation}
f(\x)=\sum_{i=1}^mc_i \x^{\a_i}-d\x^{\b},
\end{equation}
where the Newton polytope $\Delta=\New(f)$ is a lattice simplex with the vertex set $\{\a_1,\ldots,\a_m\}$, $\b$ an interior point of $\Delta$ and $c_i>0$ for $i=1,\ldots,m$. For every circuit polynomial $f$, we associate it with the {\em circuit number} defined as $\Theta_f:=\prod_{i=1}^m(c_i/\lambda_i)^{\lambda_i}$, where the $\lambda_i$'s are uniquely given by the convex combination $\b=\sum_{i=1}^m\lambda_i\a_i \textrm{ with } \lambda_i>0 \textrm{ and } \sum_{i=1}^m\lambda_i=1$. The nonnegativity of circuit polynomials is easy to check. Actually the circuit polynomial $f$ is nonnegative if and only if $\a_i\in(2\N)^n$ for all $i$, and $-\Theta_f\le d\le\Theta_f$ if $\b\notin(2\N)^n$ or $d\le\Theta_f$ if $\b\in(2\N)^n$.

If a polynomial $f$ can be written as a sum of nonnegative circuit polynomials (SONC), then $f$ is obviously nonnegative. Based on these SONC decompositions for nonnegativity certificate, new approaches were proposed for both unconstrained polynomial optimization problems and constrained polynomial optimization problems, which were proved to be significantly more efficient than the classic semidefinite programming method in many cases (\cite{dlw,diw,dkw,lw,se}).

In my previous paper \cite{wang}, it was proved that certain kinds of nonnegative polynomials decompose into a sum of nonnegative circuit polynomials with the same support. In this paper, we clarify an important fact that every SONC polynomial decomposes into a sum of nonnegative circuit polynomials with the same support. In other words, SONC decompositions for nonnegative polynomials exactly maintain sparsity of polynomials. It is dramatically unlike the case of SOS decompositions for nonnegative polynomials, in which case many extra support monomials are needed in general. This reveals the advantage of SONC decompositions for certifying nonnegativity of sparse polynomials compared with the classical SOS decompositions. By virtue of this fact, we can decide $f\in\SONC$ via relative entropy programming more efficiently.

\section{Preliminaries}
\subsection{Nonnegative Polynomials}
Let $\R[\x]=\R[x_1,\ldots,x_n]$ be the ring of real $n$-variate polynomial, and $\N^*=\N\backslash\{0\}$. For a finite set $\A\subset\N^n$, we denote by $\Conv(\A)$ the convex hull of $\A$, and by $V(\A)$ the vertices of the convex hull of $\A$. Also we denote by $V(P)$ the vertex set of a polytope $P$. We consider polynomials $f\in\R[\x]$ supported on $\A\subset\N^n$, i.e. $f$ is of the form $f(\x)=\sum_{\a\in \A}c_{\a}\x^{\a}$ with $c_{\a}\in\R, \x^{\a}=x_1^{\alpha_1}\cdots x_n^{\alpha_n}$. The support of $f$ is $\supp(f):=\{\a\in \A\mid c_{\a}\ne0\}$ and the Newton polytope is defined as $\New(f)=\Conv(\supp(f))$. For a polytope $P$, we use $P^{\circ}$ to denote the interior of $P$.

A polynomial $f\in\R[\x]$ which is nonnegative over $\R^n$ is called a {\em nonnegative polynomial}. A nonnegative polynomial must satisfy the following necessary conditions.
\begin{proposition}(\cite[Theorem 3.6]{re})\label{nc-prop2}
Let $\A\subset\N^n$ and $f=\sum_{\a\in \A}c_{\a}\x^{\a}\in\R[\x]$ with $\supp(f)=\A$. Then $f$ is nonnegative only if the following hold:
\begin{enumerate}
  \item $V(\A)\subset(2\N)^n$;
  \item If $\a\in V(\A)$, then the corresponding coefficient $c_{\a}$ is positive.
\end{enumerate}
\end{proposition}

For $f\in\R[\x]$, let $\Lambda(f):=\{\a\in\supp(f)\mid\a\in(2\N)^n\textrm{ and }c_{\a}>0\}$ and $\Gamma(f):=\supp(f)\backslash\Lambda(f)$. Then we can write $f=\sum_{\a\in\Lambda(f)}c_{\a}\x^{\a}-\sum_{\b\in\Gamma(f)}d_{\b}\x^{\b}$ with $c_{\a}>0$. The necessary conditions in Proposition \ref{nc-prop2} restate as $V(\New(f))\subseteq\Lambda(f)$.

\subsection{Sums of Nonnegative Circuit Polynomials}
A subset $\A\subseteq(2\N)^n$ is called a {\em trellis} if $\A$ comprises the vertices of a simplex.
\begin{definition}
Let $\A$ be a trellis and $f\in\R[\x]$. Then $f$ is called a {\em circuit polynomial} if it is of the form
\begin{equation}\label{nc-eq}
f(\x)=\sum_{\a\in\A}c_{\a}\x^{\a}-d\x^{\b},
\end{equation}
with $c_{\a}>0$ and $\b\in\Conv(\A)^{\circ}$. Assume
\begin{equation}
\b=\sum_{\a\in\A}\lambda_{\a}\a\textrm{ with } \lambda_{\a}>0 \textrm{ and } \sum_{\a\in\A}\lambda_{\a}=1.
\end{equation}
For every circuit polynomial $f$, we define the corresponding {\em circuit number} as $\Theta_f:=\prod_{\a\in\A}(c_{\a}/\lambda_{\a})^{\lambda_{\a}}$.
\end{definition}

The nonnegativity of a circuit polynomial $f$ is decided by its circuit number alone.
\begin{theorem}(\cite[Theorem 3.8]{iw})\label{nc-thm1}
Let $f=\sum_{\a\in\A}c_{\a} \x^{\a}-d\x^{\b}\in\R[\x]$ be a circuit polynomial and $\Theta_f$ its circuit number. Then $f$ is nonnegative if and only if $\b\notin(2\N)^n$ and $|d|\le\Theta_f$, or $\b\in(2\N)^n$ and $d\le\Theta_f$.
\end{theorem}

In analogy with writing a polynomial as sums of squares, writing a polynomial as a sum of nonnegative circuit polynomials is a certificate of nonnegativity. We denote by SONC both the class of polynomials which can be written as sums of nonnegative circuit polynomials and the property of a polynomial to be in this class. In \cite{wang}, it was proved that certain kinds of nonnegative polynomials decompose into a sum of nonnegative circuit polynomials with the same support.

Suppose $f=\sum_{\a\in\Lambda(f)}c_{\a}\x^{\a}-\sum_{\b\in\Gamma(f)}d_{\b}\x^{\b}\in\R[\x]$ with $\Gamma(f)\subset\New(f)^{\circ}$. For every $\b\in\Gamma(f)$, let
\begin{equation}
\Delta(\b):=\{\Delta\mid\Delta\textrm{ is a simplex, } \b\in\Delta^{\circ}, V(\Delta)\subseteq\Lambda(f)\}.
\end{equation}
If we can write $f=\sum_{\b\in\Gamma(f)}\sum_{\Delta\in\Delta(\b)}(\sum_{\a\in V(\Delta)}c_{\b\Delta\a}\x^{\a}-d_{\b\Delta}\x^{\b})$ such that every $\sum_{\a\in V(\Delta)}c_{\b\Delta\a}\x^{\a}-d_{\b\Delta}\x^{\b}$ is a nonnegative circuit polynomial, then we say that $f$ is a {\em sum of nonnegative circuit polynomials with the same support}.
\begin{theorem}(\cite[Theorem 3.9]{wang})\label{npgp-thm7}
Let $f=\sum_{\a\in\Lambda(f)}c_{\a}\x^{\a}-d\x^{\b}\in\R[\x]$ with $\b\in\New(f)^{\circ}$. If $f$ is nonnegative, then $f$ is a sum of nonnegative circuit polynomials with the same support.
\end{theorem}

%\begin{theorem}(\cite[Theorem 4.1]{wang})\label{npmt-thm1}
%Let $f=\sum_{\a\in\Lambda(f)}c_{\a}\x^{\a}-\sum_{\b\in\Gamma(f)}d_{\b}\x^{\b}\in\R[\x]$ with $\Gamma(f)\subset\New(f)$. Assume that all $\b$'s lie in the same side of every hyperplane determined by points among $\Lambda(f)$ and there exists a vector $\bv=(v_k)\in(\R^*)^n$ such that $d_{\b}\bv^{\b}>0$ for all $\b$. If $f$ is nonnegative, then $f$ is a sum of nonnegative circuit polynomials with no extra supports.
%\end{theorem}

\section{Nonnegative Circuit Polynomials and Sums of Squares of Binomials}
In this section, we give a connection between nonnegative circuit polynomials and sums of squares of binomials (SOSB).

We call a lattice point is {\em even} if it is in $(2\N)^n$. For a subset $M\subseteq\N^n$, define $\overline{A}(M):=\{\frac{1}{2}(\bu+\bv)\mid\bu\ne\bv,\bu,\bv\in M\cap(2\N)^n\}$ as the set of averages of distinct even points in $M$. For a trellis $\A$, we sat that $M$ is an {\em $\A$-mediated set} if $\A\subseteq M\subseteq\overline{A}(M)\cup\A$.
\begin{theorem}\label{sec3-thm1}
Let $f=\sum_{\a\in\A} c_{\a}\x^{\a}-d\x^{\b}\in\R[\x]$ be a nonnegative circuit polynomial with $\b\in\Conv(\A)^{\circ},d\ne0$. If $\b$ belongs to an $\A$-mediated set $M$, then $f$ is a sum of squares of the form $(a_{\bu}\x^{\bu}-b_{\bv}\x^{\bv})^2$, where $2\bu,2\bv\in M$.
\end{theorem}
\begin{proof}
For the proof, please refer to Theorem 5.2 in \cite{iw} which exploits Theorem 4.4 in \cite{re}.
\end{proof}

Inspired by Theorem \ref{sec3-thm1}, we are interested in the problem of deciding if there exists an $\A$-mediated set containing a given lattice point and computing one if there exists. However, there are no effective algorithms to do such thing as far as I know. On the other hand, for a trellis $\A$, there is a maximal $\A$-mediated set $\A^*$ satisfying $A(\A)\subseteq\A^*\subseteq\Conv(\A)\cup\N^n$ which contains every $\A$-mediated set. Following \cite{re}, a trellis $\A$ is called an {\em $H$-trellis} if $\A^*=\Conv(\A)\cup\N^n$. A sufficient condition for $H$-trellises is given in \cite{iw} which has the following useful corollary.
\begin{proposition}(\cite[Corollary 5.12]{iw})\label{sec3-prop}
Let $\A\subseteq\N^n$ be a trellis. Then $k\A$ is an $H$-trellis for $k\ge n$.
\end{proposition}

From Proposition \ref{sec3-prop} together with Theorem \ref{sec3-thm1}, we know that every $n$-variate nonnegative circuit polynomial supported on $k\A$ and a lattice point in the interior of $\Conv(k\A)$ is a sum of squares of binomials for $k\ge n$.
\begin{lemma}\label{sec4-lm1}
Suppose that $f(x_1,\ldots,x_n)$ is a sum of nonnegative circuit polynomials. Then $f(x_1^k,\ldots,x_n^k)$ is a sum of squares of binomials for $k\ge n$.
\end{lemma}
\begin{proof}
Assume $f=\sum f_i$, where $f_i$'s are nonnegative circuit polynomials. For $k\ge n$, since every $f_i(x_1^k,\ldots,x_n^k)$ is a sum of squares of binomials, so is $f(x_1^k,\ldots,x_n^k)$.
\end{proof}

\section{Supports of Sums of Nonnegative Circuit Polynomials}
In this section, we prove the main result of this paper: every SONC polynomial decomposes into a sum of nonnegative circuit polynomials with the same support. The proof will take use of the SOSB decompositions for SONC polynomials, so we apply the map $x_i\mapsto x_i^k$ to $f$.
\begin{lemma}\label{sec4-lm}
Let $f(x_1,\ldots,x_n)\in\R[\x]$. Then $f(x_1,\ldots,x_n)$ is a sum of nonnegative circuit polynomials with the same support if and only if $f(x_1^k,\ldots,x_n^k)$ is a sum of nonnegative circuit polynomials with the same support for an odd number $k$.
\end{lemma}
\begin{proof}
It is immediate from the fact that a polynomial $g(x_1,\ldots,x_n)$ is a nonnegative circuit polynomial if and only if $g(x_1^k,\ldots,x_n^k)$ is a nonnegative circuit polynomial for an odd number $k$.
\end{proof}

If a nonnegative polynomial $g$ has at most one negative term, i.e. $g$ has the form $\sum_{\a\in\Lambda(g)}c_{\a}\x^{\a}-d\x^{\b}$, where $\b\in(2\N)^n,d>0$ or $\b\notin(2\N)^n$, then we call $g$ a {\em banana polynomial}. By Theorem \ref{npgp-thm7}, a banana polynomial is a sum of nonnegative circuit polynomials with the same support. Moreover, if a polynomial $f$ can be written as a sum of banana polynomials, then $f\in\SONC$. For a nonnegative polynomial $f$, if we can write $f=\sum_{\b\in\Gamma(f)}(\sum_{\a\in\Lambda(f)}c_{\b\a}\x^{\a}-d_{\b}\x^{\b})$ such that every $\sum_{\a\in\Lambda(f)}c_{\b\a}\x^{\a}-d_{\b}\x^{\b}$ is a banana polynomial, then we say that $f$ is a {\em sum of banana polynomials with the same support}.
\begin{theorem}\label{sec3-thm2}
Let $f=\sum_{\a\in\Lambda(f)}c_{\a}\x^{\a}-\sum_{\b\in\Gamma(f)}d_{\b}\x^{\b}\in\R[\x]$. If $f\in\SONC$, then $f$ is a sum of nonnegative circuit polynomials with the same support.
\end{theorem}
\begin{proof}
By Lemma \ref{sec4-lm}, we only need to show that $f(x_1^{2n+1},\ldots,x_n^{2n+1})$ is a sum of nonnegative circuit polynomials with the same support. Since by Theorem \ref{npgp-thm7}, a banana polynomial is a sum of nonnegative circuit polynomials with the same support, we finish the proof by showing that $f(x_1^{2n+1},\ldots,x_n^{2n+1})$ is a sum of banana polynomials with the same support.

For simplicity, let $h=f(x_1^{2n+1},\ldots,x_n^{2n+1})$. By Theorem \ref{sec4-lm1}, we can assume $h=\sum_{i=1}^m(a_i\x^{\bu_i}-b_i\x^{\bv_i})^2$. Let us do induction on $m$. When $m=1$, $h=(a_1\x^{\bu_1}-b_1\x^{\bv_1})^2=a_1^2\x^{2\bu_1}+b_1^2\x^{2\bv_1}-2a_1b_1\x^{\bu_1+\bv_1}$ and the conclusion is obvious. Now assume that the conclusion is correct for $m-1$. Without loss of generality, assume $\bu_m+\bv_m\in\Gamma(h)$. Let $h'=\sum_{i=1}^{m-1}(a_i\x^{\bu_i}-b_i\x^{\bv_i})^2=\sum_{\a\in\Lambda(h')}c_{\a}'\x^{\a}-\sum_{\b\in\Gamma(h')}d_{\b}'\x^{\b}$. By the induction hypothesis, we can write $h'=\sum_{\b\in\Gamma(h')}(\sum_{\a\in\Lambda(h')}c_{\b\a}'\x^{\a}-d_{\b}'\x^{\b})$ as a sum of banana polynomials with the same support. Then
\begin{equation}\label{sec3-eq10}
h=\sum_{\b\in\Gamma(h')}(\sum_{\a\in\Lambda(h')}c_{\b\a}'\x^{\a}-d_{\b}'\x^{\b})+(a_m\x^{\bu_m}-b_m\x^{\bv_m})^2.
\end{equation}
From $h=h'+(a_m\x^{\bu_m}-b_m\x^{\bv_m})^2$, it follows that $\supp(h)$ and $\supp(h')$ differ among three elements: $2\bu_m,2\bv_m,\bu_m+\bv_m$. We obtain the expression of $h$ as a sum of banana polynomials with the same support from (\ref{sec3-eq10}) by adjusting the terms involving $2\bu_m,2\bv_m,\bu_m+\bv_m$ in (\ref{sec3-eq10}).

First let us consider the terms involving $2\bu_m$. If $2\bu_m\in\Gamma(h)$, then we must have $2\bu_m\in\Gamma(h')$ and $d_{2\bu_m}'>a_m^2$. By the equality $\sum_{\a\in\Lambda(h')}c_{2\bu_m\a}'\x^{\a}-d_{2\bu_m}'\x^{2\bu_m}+a_m^2\x^{2\bu_m}+b_m^2\x^{2\bv_m}-2a_mb_m\x^{\bu_m+\bv_m}=(1-\frac{a_m^2}{d_{2\bu_m}'})(\sum_{\a\in\Lambda(h')}c_{2\bu_m\a}'\x^{\a}
-d_{2\bu_m}'\x^{2\bu_m})+\sum_{\a\in\Lambda(h')}\frac{c_{2\bu_m\a}'a_m^2}{d_{2\bu_m}'}\x^{\a}+b_m^2\x^{2\bv_m}-2a_mb_m\x^{\bu_m+\bv_m}$, we obtain the expression of $h$ as $h=\sum_{\b\in\Gamma(h')\backslash\{2\bu_m\}}(\sum_{\a\in\Lambda(h')}c_{\b\a}'\x^{\a}-d_{\b}'\x^{\b})+(1-\frac{a_m^2}{d_{2\bu_m}'})(\sum_{\a\in\Lambda(h')}c_{2\bu_m\a}'\x^{\a}
-d_{2\bu_m}'\x^{2\bu_m})+\sum_{\a\in\Lambda(h')}\frac{c_{2\bu_m\a}'a_m^2}{d_{2\bu_m}'}\x^{\a}+b_m^2\x^{2\bv_m}-2a_mb_m\x^{\bu_m+\bv_m}$. If $2\bu_m\in\Lambda(h)$ and $2\bu_m\in\Gamma(h')$, then we must have $a_m^2>d_{2\bu_m}'$ and we can write $h$ as $h=\sum_{\b\in\Gamma(h')\backslash\{2\bu_m\}}(\sum_{\a\in\Lambda(h')}c_{\b\a}'\x^{\a}-d_{\b}'\x^{\b})+\sum_{\a\in\Lambda(h')}c_{2\bu_m\a}'\x^{\a}
+(a_m^2-d_{2\bu_m}')\x^{2\bu_m}+b_m^2\x^{2\bv_m}-2a_mb_m\x^{\bu_m+\bv_m}$. If $2\bu_m\notin\supp(h)$, then $2\bu_m\in\Gamma(h')$ and the terms $-d_{2\bu_m}'\x^{2\bu_m}$ and $a_m^2\x^{2\bu_m}$ must be cancelled. Hence we obtain the expression of $h$ as $h=\sum_{\b\in\Gamma(h')\backslash\{2\bu_m\}}(\sum_{\a\in\Lambda(h')}c_{\b\a}'\x^{\a}-d_{\b}'\x^{\b})+\sum_{\a\in\Lambda(h')}c_{2\bu_m\a}'\x^{\a}
+b_m^2\x^{2\bv_m}-2a_mb_m\x^{\bu_m+\bv_m}$.

Continue adjusting the terms of the expression of $h$ in a similar way for $2\bv_m$ and $\bu_m+\bv_m$. Eventually we can write $h$ as a sum of banana polynomials with the same support as desired.
\end{proof}

\begin{remark}
Theorem \ref{sec3-thm2} essentially says that the SONC decompositions for nonnegative polynomials exactly maintain sparsity of polynomials.
\end{remark}

As an application of Theorem \ref{sec3-thm2}, we give an example which is a nonnegative polynomial but not a SONC polynomial.
\begin{example}
Let $f=50x^4y^4+x^4+3y^4+800-300xy^2-180x^2y$ which is nonnegative. Let $\A=\{\a_1=(0,0),\a_2=(4,0),\a_3=(0,4),\a_4=(4,4)\}$ and $\b_1=(2,1),\b_2=(1,2)$. There are two simplexes contain $\b_1$: $\Delta_1$ with vertices $\{\a_1,\a_2,\a_3\}$ and $\Delta_2$ with vertices $\{\a_1,\a_2,\a_4\}$. There are two simplexes contain $\b_2$: $\Delta_1$ and $\Delta_3$ with vertices $\{\a_1,\a_3,\a_4\}$. If $f\in\SONC$, then by Theorem \ref{sec3-thm2}, $f$ is a sum of nonnegative circuit polynomials supported on $\Delta_1,\Delta_2,\Delta_1,\Delta_3$ respectively. Let $\b_3=(2,2),\b_4=(2,0),\b_5=(0,2)$. A $\{\a_1,\a_2,\a_3\}$-mediated set containing $\b_1$ is $\{\a_1,\a_2,\a_3,\b_1,\b_3,\b_4\}$. A $\{\a_1,\a_2,\a_4\}$-mediated set containing $\b_1$ is $\{\a_1,\a_2,\a_4,\b_1,\b_3,\b_4\}$. A $\{\a_1,\a_2\,\a_3\}$-mediated set containing $\b_2$ is $\{\a_1,\a_2,\a_3,\b_1,\b_3,\b_5\}$. A $\{\a_1,\a_3,\a_4\}$-mediated set containing $\b_2$ is $\{\a_1,\a_3,\a_4,$ $\b_1,\b_3,\b_5\}$. So by Theorem \ref{sec3-thm1}, $f$ is a sum of squares of binomials. However, in fact $f$ is even not a sum of squares. Thus $f\notin\SONC$.

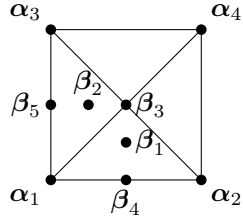
\begin{center}
\begin{tikzpicture}
\draw (0,0)--(0,2);
\draw (0,0)--(2,0);
\draw (2,0)--(2,2);
\draw (0,2)--(2,2);
\draw (0,0)--(2,2);
\draw (0,2)--(2,0);
%\draw (0,1)--(1,1);
%\draw (1,0)--(1,1);
\fill (0,0) circle (2pt);
\node[below left] (1) at (0,0) {$\a_1$};
\fill (2,0) circle (2pt);
\node[below right] (2) at (2,0) {$\a_2$};
\fill (0,2) circle (2pt);
\node[above left] (3) at (0,2) {$\a_3$};
\fill (2,2) circle (2pt);
\node[above right] (4) at (2,2) {$\a_4$};
\fill (1,0.5) circle (2pt);
\node[right] (5) at (1,0.5) {$\b_1$};
\fill (0.5,1) circle (2pt);
\node[above] (6) at (0.5,1) {$\b_2$};
\fill (1,1) circle (2pt);
\node[right] (7) at (1,1) {$\b_3$};
\fill (1,0) circle (2pt);
\node[below] (8) at (1,0) {$\b_4$};
\fill (0,1) circle (2pt);
\node[left] (9) at (0,1) {$\b_5$};
\end{tikzpicture}
\end{center}

\end{example}

%The condition that there exists a point $\bv\in(\R^*)^n$ such that $d_j\bv^{\b_j}<0$ for all $j$ in Theorem \ref{npmt-thm1} is necessary. We give an example to illustrate this.
%\begin{example}
%Let $d^*=\sup\{d\in\R_+\mid x^6y^6+x^6+y^6+1-x^2y^3-xy^2+dxy^3\textrm{ is nonnegative}\}$ and $f=x^6y^6+x^6+y^6+1-x^2y^3-xy^2+d^*xy^3$. Then $f\notin\SONC$.
%\end{example}
%\begin{proof}
%Suppose $f\in\SONC$ and let $f=\sum_{i=1}^6f_k$, where $f_1=a_1x^6y^6+c_1y^6+d_1-e_1x^2y^3,f_2=b_1x^6+c_2y^6+d_2-e_2x^2y^3,f_3=a_2x^6y^6+c_3y^6+d_3-g_1xy^2,f_4=b_2x^6+c_4y^6+d_4-g_2xy^2,f_5=a_3x^6y^6+c_5y^6+d_5+h_1xy^3,
%f_6=b_3x^6+c_6y^6+d_6+h_2xy^3$ are nonnegative circuit polynomials. Assume that $(x^*,y^*)$ is a zero of $f$. We have $f_k(x^*,y^*)=0$ for $k=1,\ldots,6$. By $f_1(x^*,y^*)=f_2(x^*,y^*)=0$, we have $y^*>0$. By $f_3(x^*,y^*)=f_4(x^*,y^*)=0$, we have $x^*>0$. By $f_5(x^*,y^*)=f_6(x^*,y^*)=0$, we have $x^*y^*>0$. It is a contradictory. Hence $f\notin\SONC$.
%\end{proof}

\section{Computation via Relative Entropy Program}
By virtue of Theorem \ref{sec3-thm2}, we can compute SONC decompositions for nonnegative polynomials via relative entropy programming (REP) more efficiently (\cite{cs,wang}). Unlike the case of SOS decompositions for nonnegative polynomials, no extra support monomials are needed.
\begin{theorem}(\cite[Theorem 3.2]{dlw})\label{crep-thm2}
Let $f=\sum_{i=1}^mc_i\x^{\a_i}-d\x^{\b}\in\R[\x]$ be a circuit polynomial, which is not a sum of monomial squares. Assume $\b=\sum_{i=1}^m\lambda_{i}\x^{\a_i}$, where $\sum_{i=1}^m\lambda_{i}=1, \lambda_{i}>0, i=1,\ldots,m$. Then $f$ is nonnegative if and only if the following REP on variables $\nu_{i}$ and $\delta_{i}$ is feasible:
\begin{equation}
\begin{cases}
\textrm{minimize} \quad\quad 1\\
\nu_{i}=d\lambda_{i}, &\textrm{for } i=1,\ldots,m\\
\nu_{i}\log(\nu_{i}/c_{i})\le \delta_{i}, &\textrm{for } i=1,\ldots,m\\
\sum_{i=1}^m\delta_{i}\le 0,
\end{cases}.
\end{equation}
\end{theorem}

We make the following assumption for the rest of this section.\\
{\bf Assumption}: Let $f=\sum_{i=1}^mc_i \x^{\a_i}-\sum_{j=1}^ld_j\x^{\b_j}\in\R[\x]$ with $\Lambda(f)=\{\a_1,\ldots,\a_m\}$ and $\Gamma(f)=\{\b_1,\ldots,\b_l\}$. For every $\b_j$, let
$$\{\Delta_{j1},\ldots,\Delta_{js_{j}}\}:=\{\Delta\mid\Delta\textrm{ is a simplex }, \b_j\in\Delta^{\circ}, V(\Delta)\subseteq\Lambda(f)\}$$
and $I_{jk}:=\{i\in[m]\mid\a_i\in V(\Delta_{jk})\}$ for $k=1,\ldots,s_j$ and $j=1,\ldots,l$. For every $\b_j$ and every $\Delta_{jk}$, since $\b_j\in\Delta_{jk}^{\circ}$, we can write $\b_j=\sum_{i\in I_{jk}}\lambda_{ijk}\a_i$, where $\sum_{i\in I_{jk}}\lambda_{ijk}=1, \lambda_{ijk}>0, i\in I_{jk}$.
\begin{theorem}\label{crep-thm1}
Let $f=\sum_{i=1}^mc_i \x^{\a_i}-\sum_{j=1}^ld_j\x^{\b_j}\in\R[\x]$ with $\Lambda(f)=\{\a_1,\ldots,\a_m\}$, $\Gamma(f)=\{\b_1,\ldots,\b_l\}$ and $V(\New(f))\subseteq\Lambda(f)$, which is not a sum of monomial squares. Then $f\in\SONC$ if and only if the following REP on variables $d_{jk}$, $\nu_{ijk}$, $c_{ijk}$ and $\delta_{ijk}$ is feasible:
\begin{equation}\label{crep-eq1}
\begin{cases}
\textrm{minimize} \quad\quad 1\\
\nu_{ijk}=d_{jk}\lambda_{ijk}, &\textrm{for } i\in I_{jk},k=1,\ldots,s_j,j=1,\ldots,l\\
\nu_{ijk}\log(\nu_{ijk}/c_{ijk})\le \delta_{ijk}, &\textrm{for } i\in I_{jk},k=1,\ldots,s_j,j=1,\ldots,l\\
\sum_{i\in I_{jk}}\delta_{ijk}\le 0, &\textrm{for } k=1\ldots,s_j,j=1,\ldots,l\\
\sum_{j=1}^l\sum_{i\in I_{jk}} c_{ijk}=c_i, &\textrm{for } i=1,\ldots,m\\
\sum_{k=1}^{s_j}d_{jk}=d_j, &\textrm{for } j=1,\ldots,l
\end{cases}.
\end{equation}
\end{theorem}
\begin{proof}
Suppose $f_{jk}=\sum_{i\in I_{jk}}c_{ijk}\x^{\a_i}-d_{jk}\x^{\b_j}$ is a nonnegative circuit polynomial for $k=1,\ldots,s_j,j=1,\ldots,l$ and $f=\sum_{j=1}^l\sum_{k=1}^{r}f_{jk}$. Then by Theorem \ref{crep-thm2}, $(d_{jk})_{j,k}$, $(\nu_{ijk})_{i,j,k}=(d_{jk}\lambda_{ijk})_{i,j,k}$, $(c_{ijk})_{i,j,k}$ and $(\delta_{ijk})_{i,j,k}=(\nu_{ijk}\log(\nu_{ijk}/c_{ijk}))_{i,j,k}$ is a feasible solution of (\ref{crep-eq1}).

Conversely, suppose that $(d_{jk})_{j,k}$, $(\nu_{ijk})_{i,j,k}$, $(c_{ijk})_{i,j,k}$ and $(\delta_{ijk})_{i,j,k}$ is a feasible solution of (\ref{crep-eq1}). Let $f_{jk}=\sum_{i\in I_{jk}}c_{ijk}\x^{\a_i}-d_{jk}\x^{\b_j}$ for $k=1,\ldots,s_j,j=1,\ldots,l$. Then by Theorem \ref{crep-thm2}, $f_{jk}$ is a nonnegative circuit polynomial for all $j,k$. Moreover, by the last two equality conditions in (\ref{crep-eq1}), we have $f=\sum_{j=1}^l\sum_{k=1}^{r}f_{jk}$. Thus, $f\in\SONC$.
\end{proof}

\bibliographystyle{amsplain}

\end{document}